\documentclass[11pt]{amsart}

\usepackage{amsmath}
\usepackage{amssymb}
\usepackage{graphicx}

\usepackage{amsmath,amsfonts,amssymb,amscd,amsthm,amsbsy,graphics,epsf}
\usepackage{comment,graphicx}

\newtheorem{theorem}{Theorem}[section]
\newtheorem{lemma}[theorem]{Lemma}

\theoremstyle{definition}
\newtheorem{definition}[theorem]{Definition}

\theoremstyle{remark}

\numberwithin{equation}{section}

\begin{document}

\title[Averaging one-point hyperbolic-type metrics]{Averaging one-point hyperbolic-type metrics}

\author[A. Aksoy]{Asuman G\"{u}ven Aksoy}
\address{Department of Mathematics\\ Claremont McKenna College\\ Claremont, CA, 91711}
\email{aaksoy@cmc.edu}
\author[Z. Ibragimov]{Zair Ibragimov}
\address{Department of Mathematics\\ California State University at Fullerton\\ Fullerton, CA, 92831}
\email{zibragimov@fullerton.edu}
\author[W. Whiting]{Wesley Whiting}
\address{Department of Mathematics\\ California State University at Fullerton\\ Fullerton, CA, 92831}
\email{weswhiting@gmail.com}
\keywords{Averages, Cassinian metrics, one-point metrics, hyperbolic-type metrics, Gromov hyperbolicity}
\subjclass[2000]{Primary 30F45; Secondary 51F99, 30C99}

\begin{abstract}
It is known that the $\tilde j$-metric, half-apollonian metric and scale-invariant Cassinian metric are not Gromov hyperbolic. These metrics are defined as a supremum of {\it{one-point}} metrics (i.e., metrics constructed using one boundary point) and the supremum is taken over all boundary points. The aim of this paper is to show that taking the average instead of the supremum yields a metric that preserves the Gromov hyperbolicity. Moreover, we show that the Gromov hyperbolicity constant of the resulting metric does not depend on the number of metrics used in taking the average.  We also provide an example to show that the average of Gromov hyperbolic metrics is not, in general, Gromov hyperbolic.

\end{abstract}

\maketitle

\section{Introduction}\label{Intro}

The hyperbolic metric has been a powerful tool in planar complex analysis. In higher dimensional Euclidean spaces, the hyperbolic metric exists only in balls and half-spaces and the lack of hyperbolic metric in general domains has been a primary motivation for introducing the so-called {\it{hyperbolic-type}} metrics in geometric function theory. Examples of such metrics include $\tilde j$-metric, Apollonian metric, Seittenranta's metric, half-apollonian metric, scale-invariant Cassinian metric and M\"obius-invariant Cassinian metric (see \cite{Bea98,Has06,HIL06,HL,Ibr1,Ib1,Ib2,L07,Sei99,Vuo88} and the references therein). All these metrics are so-called {\it{point-distance metrics}} meaning that they are defined in terms of distance functions and can be classified into {\it{one-point}} metrics or {\it{two-point}} metrics based on the number of boundary points used in their definitions. For example, the Apollonian, Seittenranta and the M\"obius-invariant Cassinian metrics are two-point, point-distance metrics while the $\tilde j$-metric, half-apollonian metric and the scale-invariant Cassinian metric are all one-point, point-distance metrics. 

One of the key features of hyperbolic-type metrics is the Gromov hyperbolicity property. The Apollonian, Seittenranta and M\"obius-invariant Cassinian metrics are roughly similar to each other and, in particular, they are all Gromov hyperbolic (see \cite[Theorem 4.8 and Theorem 5.4]{Ib2}). The $\tilde j$-metric, half-apollonian and scale-invariant Cassinian metrics are also roughly similar to each other. However, they are Gromov hyperbolic if and only if the underlying domain has only one boundary point (\cite[Theorem 3.3 and Theorem 3.5]{Ib1}). In other words, the one-point versions of these metrics are Gromov hyperbolic but the metrics themselves, defined as the supremums of their one-point versions, are not Gromov hyperbolic.

We briefly mention a general approach to constructing one-point hyperbolic-type metrics in the settings of Euclidean spaces. Let $D\subset\mathbb R^n$ be any domain with non-empty boundary $\partial D$. To construct a one-point hyperbolic-type metric $d_D$ on $D$, one first constructs a Gromov hyperbolic metric $d_p$ on the one-punctured space $\mathbb R^n\setminus\{p\}$ for each $p\in\mathbb R^n$ and then defines $d_D$ by $d_D(x,y)=\sup\{d_p(x,y)\colon\, p\in\partial D\}$. Taking a supremum in this context is very natural since the boundary $\partial D$ is usually uncountable. However, as it turns out, the Gromov hyperbolicity of $d_p$ is not preserved when taking the supremum.

In this paper we propose an alternative approach to constructing a metric from the one-point metrics mentioned above. Namely, we propose to take the average of these one-point metrics instead of taking their supremum. As mentioned above, these metrics are roughly similar to each other and hence so are their averages. Therefore, here we only consider the one-point scale-invariant Cassinian metrics. The main result of this paper states that the average of finitely many, one-point scale-invariant Cassinian metrics is Gromov hyperbolic and, more importantly, its Gromov hyperbolicity constant does not depend on the number of metrics (Lemma~\ref{one-point} and Theorem~\ref{T1}). Even though we consider here the averages of finitely many metrics, the Gromov hyperbolicity constant being independent of the number of metrics allows one to consider domains which are the complements of certain self-similar sets (\cite{H}). %Lastly, we provide an example to show that the average of two Gromov hyperbolic metrics is not, in general, Gromov hyperbolic (Lemma~\ref{sum of metrics is not GH}). 

To the best of our knowledge, averaging one-point metrics has not been considered before. However, germs of this idea can be traced back to the work of F.W. Gehring and B. Osgood. More specifically, let $D$ be a proper subdomain of $\mathbb R^n$. Then the $j_D$-metric (see, \cite[p. 51]{GO79}), 
$$
j_D(x,y)=\frac {1}{2}\Big[\log\Big(1+\frac {|x-y|}{\operatorname{dist}(x,\partial D)}\Big)+\log\Big(1+\frac {|x-y|}{\operatorname{dist}(y,\partial D)}\Big)\Big],
$$
which is defined as average, is Gromov hyperbolic (\cite[Theorem 1]{Has06}). As mentioned above, the $\tilde j_D$-metric, 
$$
\tilde j_D(x,y)=\sup\Big\{\log\Big(1+\frac {|x-y|}{\operatorname{dist}(x,\partial D)}\Big),\ \ \log\Big(1+\frac {|x-y|}{\operatorname{dist}(y,\partial D)}\Big)\Big\},
$$
which is defined as supremum, is not Gromov hyperbolic (\cite[Theorem 3]{Has06}). (Note that in \cite{Has06} the author denotes the $j$-metric by $\tilde j$ and the $\tilde j$-metric by $j$).

To formulate the main results of the paper, let $(X,d)$ be arbitrary metric space. For each $p\in X$, we define a distance function $\tau_p$ on $X\setminus\{p\}$, by
\begin{equation}\label{aaEq2.1}
\tau_p(x,y)=\log\Big(1+2\frac {d(x,y)}{\sqrt{d(x,p)d(y,p)}}\Big).
\end{equation}
For $p_1,p_2,\dots,p_k \in  X$ and $D= X\setminus\{p_1,p_2,\dots,p_k\}$, we define a metric $\tau_D$ on $D$ by taking the simple average of the metrics $\tau_{p_i}$, namely,
$$
\hat\tau_D(x,y)=\frac {1}{k}\sum_{i=1}^{k}\tau_{p_i}(x,y)=\frac {1}{k}\sum_{i=1}^{k}\log\Big(1+ 2\frac {d(x,y)}{\sqrt{d(x,p_i)d(y,p_i)}}\Big).
$$
We prove that for each $p\in X$, the metric $\tau_p$ is Gromov hyperbolic with $\delta=\log 3+\log 2$ (Lemma~\ref{one-point}) and that for any $k\geq 1$, the metric $\hat\tau_D(x,y)$ is Gromov hyperbolic with $\delta= 3 \log 3+2 \log 2$ (Theorem~\ref{T1}).  The latter is an unexpected result since we also provide an example to demonstrate that the average of  two Gromov-hyperbolic metrics is not necessarily Gromov-hyperbolic (Lemma~\ref{sum of metrics is not GH}).

%%%%%%%%%%%%%%%%%%%%%%%%%%%%%%%%%%%%%%%%%%%%%%%%%%%%%%%%%%%%%%%%%%%%%%%%%%%%%%%%%

\section{One-point, scale-invariant Cassinian metric on general metric spaces}\label{Cassin}

In this section we define one-point, scale-invariant Cassinian metrics in the context of arbitrary metric spaces and in Section~\ref{Gromov Hyperbolic} we study Gromov hyperbolicity of the average of finitely many such metrics.
Let $(X,d)$ be an arbitrary metric space. For each $p\in X$, we define a distance function $\tau_p$ on $X\setminus\{p\}$ by
\begin{equation}\label{aaEq2.1}
\tau_p(x,y)=\log\Big(1+2\frac {d(x,y)}{\sqrt{d(x,p)d(y,p)}}\Big).
\end{equation}

\begin{theorem}\label{GH}
Let $(X,d)$ be an arbitrary metric space and let $p\in X$ be an arbitrary point. Then the distance function $\tau_p$ is a metric on $X\setminus\{p\}$. 
\end{theorem}
\begin{proof}
Clearly, $\tau_p(x,y)\geq 0$, $\tau_p(x,y)=\tau_p(y,x)$ and $\tau_p(x,y)=0$ if and only if $x=y$. So it is enough to show that the triangle inequality holds. That is,
\begin{equation}\label{eq3.2}
\tau_p(x,y)\leq\tau_p(x,z)+\tau_p(z,y)
\end{equation}
for all $x,y,z\in D$. Inequality (\ref{eq3.2}) is equivalent to
$$
\frac {d(x,y)}{\sqrt{d(x,p)d(y,p)}}\leq\frac {d(x,z)}{\sqrt{d(x,p)d(z,p)}}+\frac {d(z,y)}{\sqrt{d(z,p)d(y,p)}}+2\frac {d(x,z)d(z,y)}{d(z,p)\sqrt{d(x,p)d(y,p)}}
$$
or, equivalently,
\begin{equation}
\frac {d(x,y)d(z,p)}{d(x,z)d(y,z)}\leq\frac {\sqrt{d(x,p)d(z,p)}}{d(x,z)}+\frac {\sqrt{d(y,p)d(z,p)}}{d(y,z)}+2.
\end{equation}
Since 
$$
\frac {d(x,y)d(z,p)}{d(x,z)d(y,z)}\leq\frac {d(y,z)d(z,p)}{d(x,z)d(y,z)}+\frac {d(x,z)d(z,p)}{d(x,z)d(y,z)}=\frac {d(z,p)}{d(x,z)}+\frac {d(z,p)}{d(y,z)},
$$
it suffices to show that
$$
\frac {d(z,p)}{d(x,z)}\leq\frac {\sqrt{d(x,p)d(z,p)}}{d(x,z)}+1\quad\text{and}\quad\frac {d(z,p)}{d(y,z)}\leq\frac {\sqrt{d(y,p)d(z,p)}}{d(y,z)}+1.
$$
Due to symmetry, it suffices to prove the first inequality. If $d(z,p)\leq d(x,p)$, then
$$
\frac {d(z,p)}{d(x,z)}\leq\frac {\sqrt{d(x,p)d(z,p)}}{d(x,z)}<\frac {\sqrt{d(x,p)d(z,p)}}{d(x,z)}+1.
$$
If $d(x,p)\leq d(z,p)$. Then
$$
\frac {d(z,p)}{d(x,z)}\leq\frac {d(x,z)+d(x,p)}{d(x,z)}\leq\frac {d(x,z)+\sqrt{d(x,p)d(z,p)}}{d(x,z)}=\frac {\sqrt{d(x,p)d(z,p)}}{d(x,z)}+1,
$$
completing the proof.
\end{proof}

One can easily see that for all $x,y\in X\setminus\{p\}$ we have
\begin{equation}\label{aaaEq2.2}
\tilde\tau_p(x,y)\leq\tau_p(x,y)\leq\tilde\tau_p(x,y)+\log 2.
\end{equation}
Here
\begin{equation}\label{aaEq2.2}
\tilde\tau_p(x,y)=\log\Big(1+\frac {d(x,y)}{\sqrt{d(x,p)d(y,p)}}\Big)=\log\frac {\mu_p(x,y)}{\sqrt{d(x,p)d(y,p)}}.
\end{equation}
The distance function $\tilde\tau_p$ was introduced and studied in the context of Euclidean spaces in (\cite{Ib1}), where it was referred to as one-point, scale-invariant Cassinian metric. However, $\tilde\tau_p$ is not a metric in the context of general metric spaces. Indeed, let $X=\{p,x,y,z\}$ and define $d(p,x)=d(y,z)=2$, $d(p,y)=d(p,z)=d(x,y)=d(x,z)=1$. Clearly, $d$ is a metric on $X$. One can easily see that $\tilde\tau_p(y,z)>\tilde\tau_p(x,y)+\tilde\tau_p(x,z)$. Therefore, $\tilde\tau_p$ is not a metric on $X\setminus\{p\}$ justifying the introduction of its modified version $\tau_p$. However, it turns out that, if $(X,d)$ is a {\it{Ptolemaic}} metric space, then $\tilde\tau_p$ is a metric on $X\setminus\{p\}$ for each $p\in X$. Recall that a metric space $(X,d)$ is called Ptolemaic if
\begin{equation}\label{aaaaEq2.2}
d(x,y)d(z,w)\leq d(x,z)d(y,w)+d(x,w)d(y,z)
\end{equation}
for all $x,y,z,w\in X$.

\begin{theorem}\label{3.1}
Let $(X,d)$ be a Ptolemaic metric space and let $p\in X$ be an arbitrary point. Then the distance function $\tilde\tau_p$ is a metric on $X\setminus\{p\}$. 
\end{theorem}
\begin{proof}
Clearly, it is enough to show that the triangle inequality holds. That is,
\begin{equation}\label{eq3.22}
\tilde\tau_p(x,y)\leq\tilde\tau_p(x,z)+\tilde\tau_p(z,y)
\end{equation}
for all $x,y,z\in X\setminus\{p\}$. Inequality (\ref{eq3.22}) is equivalent to
\begin{equation*}\label{eq3.3}
\Big(1+\frac {d(x,y)}{\sqrt{d(x,p)d(y,p)}}\Big)\leq\Big(1+\frac {d(x,z)}{\sqrt{d(x,p)d(z,p)}}\Big)\Big(1+\frac {d(z,y)}{\sqrt{d(z,p)d(y,p)}}\Big),
\end{equation*}
which is equivalent to
\begin{equation}\label{eq3.4}
\begin{split}
\frac {d(x,y)}{\sqrt{d(x,p)d(y,p)}}&\leq\frac {d(x,z)}{\sqrt{d(x,p)d(z,p)}}+\\
&+\frac {d(z,y)}{\sqrt{d(z,p)d(y,p)}}+\frac {d(x,z)d(z,y)}{d(z,p)\sqrt{d(x,p)d(y,p)}}.
\end{split}
\end{equation}
Without loss of generality we can assume that $d(x,p)\leq d(y,p)$. 

If $d(z,p)\leq d(x,p)\leq d(y,p)$, then 
$$
\sqrt{d(x,p)d(y,p)}\geq \sqrt{d(x,p)d(z,p)}\qquad\text{and}\qquad \sqrt{d(x,p)d(y,p)}\geq \sqrt{d(z,p)d(y,p)}.
$$
By the triangle inequality we then obtain
\begin{equation*}
\begin{split}
\frac {d(x,y)}{\sqrt{d(x,p)d(y,p)}}&\leq\frac {d(x,z)}{\sqrt{d(x,p)d(y,p)}}+\frac {d(z,y)}{\sqrt{d(x,p)d(y,p)}}\\
&\leq\frac {d(x,z)}{\sqrt{d(x,p)d(z,p)}}+\frac {d(z,y)}{\sqrt{d(z,p)d(y,p)}},
\end{split}
\end{equation*}
establishing (\ref{eq3.4}). 

If $d(x,p)\leq d(y,p)\leq d(z,p)$, then
$$
d(z,p)d(x,y)\leq d(y,p)d(x,z)+d(x,p)d(z,y)
$$
by the Ptolemy's Inequality. Since $d(x,p)\leq d(z,p)$ and $d(y,p)\leq d(z,p)$, we have 
$$
d(x,p)\leq\sqrt{d(x,p)d(z,p)}\qquad\text{and}\qquad d(y,p)\leq\sqrt{d(y,p)d(z,p)}. 
$$
Hence
$$
d(z,p)d(x,y)\leq \sqrt{d(y,p)d(z,p)}d(x,z)+\sqrt{d(x,p)d(z,p)}d(z,y).
$$
Consequently,
$$
\frac {d(x,y)}{\sqrt{d(x,p)d(y,p)}}\leq\frac {d(x,z)}{\sqrt{d(x,p)d(z,p)}}+\frac {d(z,y)}{\sqrt{d(z,p)d(y,p)}},
$$
establishing (\ref{eq3.4}).

Finally, if $d(x,p)\leq d(z,p)\leq d(y,p)$, then $d(x,p)\leq\sqrt{d(x,p)d(z,p)}$ since $d(x,p)\leq d(z,p)$. By the triangle inequality we have $d(z,p)\leq d(x,p)+d(x,z)$. Hence 
$$
d(z,p)\leq\sqrt{d(x,p)d(z,p)}+d(x,z),
$$
or, equivalently,
$$
\frac {1}{\sqrt{d(x,p)}}\leq\frac {1}{\sqrt{d(z,p)}}+\frac {d(x,z)}{d(z,p)\sqrt{d(x,p)}}.
$$
Thus,
\begin{equation}\label{eq3.5}
\frac {d(z,y)}{\sqrt{d(x,p)d(y,p)}}\leq\frac {d(z,y)}{\sqrt{d(z,p)d(y,p)}}+\frac {d(x,z)d(z,y)}{d(z,p)\sqrt{d(x,p)d(y,p)}}.
\end{equation}
Now by the triangle inequality we have
\begin{equation}\label{eq3.6}
\frac {d(x,y)}{\sqrt{d(x,p)d(y,p)}}\leq\frac {d(x,z)}{\sqrt{d(x,p)d(y,p)}}+\frac {d(z,y)}{\sqrt{d(x,p)d(y,p)}}
\end{equation}
Also, since $d(z,p)\leq d(y,p)$, we have
\begin{equation}\label{eq3.7}
\frac {d(x,z)}{\sqrt{d(x,p)d(y,p)}}\leq\frac {d(x,z)}{\sqrt{d(x,p)d(z,p)}}.
\end{equation}
Therefore, combining inequalities (\ref{eq3.5}), (\ref{eq3.6}) and (\ref{eq3.7}), we see that inequality (\ref{eq3.4}) holds also in this case. The proof is complete.
\end{proof}

\begin{definition}
In the context of a general metric space $(X,d)$, the metrics $\tau_p$, $p\in X$, are called one-point, scale-invariant Cassinian metrics. 
\end{definition}

%%%%%%%%%%%%%%%%%%%%%%%%%%%%%%%%%%%%%%%%%%%%%%%%%%%%%%%%%%%%%%%%%%%%%%%%%%%%%%%%%

\section{Technical results}\label{T}

In this section we establish several results needed in Section~\ref{Gromov Hyperbolic}. Throughout this section we let $(X,d)$ be an arbitrary metric space. Fix a point $p\in X$ and define
$$
\mu_p(x,y)=d(x,y)+\sqrt{d(x,p)d(y,p)}\qquad\text{for}\quad x,y\in X.
$$
In this section we study some properties of $\mu_p$, especially Lemma~\ref{lem0} and Lemma~\ref{lem3}, which will be needed in Section~\ref{Gromov Hyperbolic}. In what follows, we set 
$$
a\wedge b=\min\{a,b\}\qquad\text{and}\qquad a\vee b=\max\{a,b\}
$$ 
for non-negative real numbers $a$ and $b$. Observe that 
\begin{equation}\label{max}
(a\vee b)(c\vee d)=ac\vee ad\vee bc\vee bd
\end{equation}
for all non-negative real numbers $a,b,c,d$. 

\begin{lemma}\label{lem0}
For all $x,y,z,w\in X$ we have
 \begin{equation}\label{5.4}
\mu_p(x,y)\mu_p(z,w)\leq 9\big[\mu_p(x,z)\mu_p(y,w)\vee\mu_p(x,w)\mu_p(y,z)\big]
\end{equation}
\end{lemma}
\begin{proof}
Since $d(x,y)\leq d(x,p)+d(y,p)\leq 2(d(x,p)\vee d(y,p))$ and since
$$
\sqrt{d(x,p)d(y,p)}\leq \frac {d(x,p)+d(y,p)}{2}\leq d(x,p)\vee d(y,p),  
$$
we have
\begin{equation}\label{1.2}
\mu_p(x,y)\leq \frac {3}{2}\big[d(x,p)+d(y,p)\big]\leq 3\big[d(x,p)\vee d(y,p)\big]
\end{equation}
for all $x,y\in X$. Also, since $d(x,y)\geq d(x,p)\vee d(y,p)-d(x,p)\wedge d(y,p)$ and since
$\sqrt{d(x,p)d(y,p)}\geq d(x,p)\wedge d(y,p)$,
we have 
\begin{equation}\label{1.31}
\mu_p(x,y)\geq d(x,p)\vee d(y,p)\geq \frac{1}{2}\big[d(x,p)+d(y,p)\big]
\end{equation}
for all $x,y\in X$.
Using (\ref{max}), (\ref{1.2}) and (\ref{1.31})  we have
\begin{equation*}
\begin{split}
\frac {1}{9}&\mu_p(x,y)\mu_p(z,w)\leq \big[d(x,p)\vee d(y,p)\big]\big[d(z,p)\vee d(w,p)\big]\\
&=d(x,p)d(z,p)\vee d(x,p)d(w,p)\vee d(y,p)d(z,p)\vee d(y,p)d(w,p)\\
&\leq \Big[d(x,p)d(y,p)\vee d(x,p)d(w,p)\vee d(z,p)d(y,p)\vee d(z,p)d(w,p)\Big]\vee \\
&\vee \Big[d(x,p)d(y,p)\vee d(x,p)d(z,p)\vee d(w,p)d(y,p)\vee d(w,p)d(z,p)\Big]\\
&=\Big[\big(d(x,p)\vee d(z,p)\big)\big(d(y,p)\vee d(w,p)\big)\Big]\vee\Big[\big(d(x,p)\vee d(w,p)\big)\big(d(y,p)\vee d(z,p)\big)\Big]\\
&\leq\mu_p(x,z)\mu_p(y,w)\vee \mu_p(x,w)\mu_p(y,z),
\end{split}
\end{equation*}
as required.
\end{proof}

Note that 
\begin{equation}\label{1.5}
\mu_p(x,z)+\mu_q(y,z)\geq d(x,z)+d(y,z)\geq d(x,y)
\end{equation}
for all $x,y,z,q\in X$. In particular, for all $x,y,z,q\in X$, we have
\begin{equation}\label{1.6}
\mu_p(x,z)\vee\mu_q(y,z)\geq \frac {1}{2}d(x,y).
\end{equation}

\begin{lemma}\label{lem1}
Let $x,y,z\in X$ be arbitrary points. If
$$
\mu_p(x,z)\vee\mu_p(y,z)\geq K\big[\mu_p(x,z)\wedge\mu_p(y,z)\big]
$$
for some $K>3$, then
$$
\mu_p(x,z)+\mu_p(y,z)\leq\frac {3(K+3)}{2(K-3)}d(x,y).
$$
%and
%$$
%\mu_p(x,z)\vee\mu_p(y,z)\geq\frac {K}{K+3}|x-y|.
%$$
\end{lemma}
\begin{proof}
Without loss of generality we can assume that $\mu_p(x,z)\geq\mu_p(y,z)$. Using (\ref{1.31}) we obtain
$$
\frac {K}{2}\big(d(y,p)+d(z,p)\big)\leq K\mu_{p}(y,z)\leq \mu_{p}(x,z)\leq \frac {3}{2}\big(d(x,p)+d(z,p)\big),
$$
which implies $Kd(y,p)+(K-3)d(z,p)\leq 3d(x,p)$.
In particular,
$$
2d(z,p)\leq\frac {6}{K-3}d(x,p)-\frac {2K}{K-3}d(y,p).
$$
The latter along with (\ref{1.2}) implies
\begin{equation*}\label{L3_3}
\begin{split}
\mu_{p}(x,z)+\mu_{p}(y,z)&\leq\frac {3}{2} \big(d(x,p)+d(y,p)+2d(z,p)\big)\\
&\leq \frac {3}{2}\big(d(x,p)+d(y,p)+\frac {6}{K-3}d(x,p)-\frac {2K}{K-3}d(y,p)\big)\\
&=\frac {3(K+3)}{2(K-3)}\big(d(x,p)-d(y,p)\big)\leq \frac {3(K+3)}{2(K-3)}d(x,y),
\end{split}
\end{equation*}
completing the proof.
\end{proof}

Suppose now that $p_1,p_2,\dots,p_k$ are arbitrary points in $X$ and set $P=\{p_1,p_2,\dots,p_k\}$.

\begin{lemma}\label{L3}
For all $x,y,z\in X$ we have
\begin{equation}\label{L3_01}
\prod_{i=1}^{k}\Big(\mu_{p_i}(x,z)+\mu_{p_i}(y,z)\Big)\leq 9^k\Bigg(\prod_{i=1}^{k}\mu_{p_i}(x,z)+\prod_{i=1}^{k}\mu_{p_i}(y,z)\Bigg).
\end{equation}
\end{lemma}
\begin{proof}
Let $x,y,z\in X$ be arbitrary points. For simplicity, we set
$$
a_i=\mu_{p_i}(x,z)\quad\text{and}\quad b_i=\mu_{p_i}(y,z),\ \ i=1,2,\dots,k.
$$ 
By (\ref{1.6}) we then have
\begin{equation}\label{L3_1.6}
a_i\vee b_j\geq\frac {1}{2}d(x,y)\qquad\text{for\ all}\quad i,j=1,2,\dots,k.
\end{equation}
We will prove the lemma by induction. So assume first that $k=2$.
Hence we need to show that
\begin{equation}\label{L3_1}
(a_1+b_1)(a_2+b_2)\leq 81(a_1a_2+b_1b_2).
\end{equation}
%If $a_1\geq b_1$ and $a_2\geq b_2$ or $a_1\leq b_1$ and $a_2\leq b_2$, then 
%$0\leq (a_1-b_1)(a_2-b_2)$ and hence $a_1b_2+b_1a_2\leq a_1a_2+b_1b_2$. The latter implies $(a_1+b_1)(a_2+b_2)\leq 2(a_1a_2+b_1b_2)$
%$$
%(a_1+b_1)(a_2+b_2)\leq 4(a_1\vee b_1)(a_2\vee b_2)=4(a_1a_2\vee b_1b_2)\leq 4(a_1a_2+b_1b_2)
%$$
%so that (\ref{L3_1}) holds. Hence we assume that either $a_1\geq b_1$ and $a_2\leq b_2$ or $a_1\leq b_1$ and $a_2\geq b_2$. 

{\it{Case 1:}} $a_1\vee b_1\leq 6(a_1\wedge b_1)$ or $a_2\vee b_2\leq 6(a_2\wedge b_2)$. Without loss of generality we can assume that $a_1\vee b_1\leq 6(a_1\wedge b_1)$. Then 
$$
a_1+b_1=a_1\vee b_1+a_1\wedge b_1\leq 7(a_1\wedge b_1)\quad\text{and}\quad (a_1\wedge b_1)(a_2+b_2)\leq a_1a_2+b_1b_2.
$$
Hence
$$
(a_1+b_1)(a_2+b_2)\leq 7(a_1\wedge b_1)(a_2+b_2)\leq 7(a_1a_2+b_1b_2)
$$
so that (\ref{L3_1}) holds in this case.  

{\it{Case 2:}} $a_1\vee b_1\geq 6(a_1\wedge b_1)$ and $a_2\vee b_2\geq 6(a_2\wedge b_2)$. 
Without loss of generality we can assume that $a_1=a_1\wedge b_1\wedge a_2\wedge b_2$. By (\ref{L3_1.6}) we then have
$$
b_1\geq\frac {1}{2}d(x,y)\quad\text{and}\quad b_2\geq\frac {1}{2}d(x,y).
$$
Hence
$$
a_1a_2+b_1b_2\geq b_1b_2\geq\frac {1}{4}\big[d(x,y)\big]^2.
$$
Also, by Lemma~\ref{lem1} we have
$$
a_1+b_1\leq\frac {9}{2}d(x,y)\quad\text{and}\quad a_2+b_2\leq\frac {9}{2}d(x,y)
$$
and hence
$$
(a_1+b_1)(a_2+b_2)\leq\frac {81}{4}\big[d(x,y)\big]^2.
$$
Consequently,
$$
(a_1+b_1)(a_2+b_2)\leq\frac {81}{4}\big[d(x,y)\big]^2\leq 81(a_1a_2+b_1b_2),
$$
completing the proof of the lemma for $k=2$.

Assume now that (\ref{L3_01}) holds for $k=m$. That is,
\begin{equation}\label{L3_9}
\prod_{i=1}^{m}(a_i+b_i)\leq 9^m\Big(\prod_{i=1}^{m}a_i+\prod_{i=1}^{m}b_i\Big).
\end{equation}
We need to show that it also holds for $k=m+1$. That is,
\begin{equation}\label{L3_10}
\prod_{i=1}^{m+1}(a_i+b_i)\leq {9}^{m+1}\Big(\prod_{i=1}^{m+1}a_i+\prod_{i=1}^{m+1}b_i\Big).
\end{equation}

{\it{Case 1:}} $a_i\vee b_i\leq 6 (a_i\wedge b_i)$ for some $i\in\{1,2,\dots,m+1\}$. Note that
$$
a_i+b_i=(a_i\vee b_i)+(a_i\wedge b_i)\leq 7(a_i\wedge b_i).
$$
Without loss of generality we can assume that $i=1$. Then
$$
\prod_{i=1}^{m+1}a_i+\prod_{i=1}^{m+1}b_i\geq (a_1\wedge b_1)\Big(\prod_{i=2}^{m+1}a_i+\prod_{i=2}^{m+1}b_i\Big)
$$
and hence
\begin{equation*}
\begin{split}
\prod_{i=1}^{m+1}(a_i+b_i)&=(a_1+b_1)\prod_{i=2}^{m+1}(a_i+b_i)\leq (a_1+b_1){9}^m\Big(\prod_{i=2}^{m+1}a_i+\prod_{i=2}^{m+1}b_i\Big)\\
&\leq7(a_1\wedge b_1){9}^m\Big(\prod_{i=2}^{m+1}a_i+\prod_{i=2}^{m+1}b_i\Big)< {9}^{m+1}\Big(\prod_{i=1}^{m+1}a_i+\prod_{i=1}^{m+1}b_i\Big),
\end{split}
\end{equation*}
as required.

{\it{Case 2:}} $a_i\vee b_i\geq 6 (a_i\wedge b_i)$ for all $i\in\{1,2,\dots,m+1\}$.
Without loss of generality we can assume that $a_1$ is the smallest of the numbers $a_i$ and $b_i$ for all $i=1,2,\dots,{m+1}$. By (\ref{L3_1.6}) we then have
$$
b_i\geq\frac {1}{2}d(x,y)\quad\text{for\ all}\ \ i=1,2,\dots,{m+1}.
$$
Hence
$$
\prod_{i=1}^{m+1}a_i+\prod_{i=1}^{m+1}b_i\geq\prod_{i=1}^{m+1}b_i \geq\frac {1}{2^{m+1}}\big[d(x,y)\big]^{m+1}.
$$
Also, by Lemma~\ref{lem1} we have $a_i+b_i\leq (9/2)d(x,y)$ for each $i$. Hence
$$
\prod_{i=1}^{m+1}(a_i+b_i)\leq\Big(\frac {9}{2}\Big)^{m+1}\big[d(x,y)\big]^{m+1}.
$$
Consequently,
$$
\prod_{i=1}^{m+1}(a_i+b_i)\leq\Big(\frac {9}{2}\Big)^{m+1}\big[d(x,y)\big]^{m+1}\leq {9}^{m+1}\Big(\prod_{i=1}^{m+1}a_i+\prod_{i=1}^{m+1}b_i\Big)
$$
completing the proof of the lemma.
\end{proof}

We need the following lemma. For $K=1$, this lemma was proved in \cite{Ibr4} (see \cite[Lemma 3.7]{Ibr4}).

\begin{lemma}\label{lem4}
Let $r_{ij}\geq 0$ be real numbers such that 
$r_{ij}=r_{ji}$ and $r_{ij}\leq K(r_{ik}+r_{jk})$ for some $K\geq 1$ and for all $i,j,k\in\{1,2,3,4\}$. Then 
$\sqrt{r_{12}r_{34}}\leq K(\sqrt{r_{13}r_{24}}+\sqrt{r_{14}r_{23}})$.
In particular, 
$$
r_{12}r_{34}\leq 2K^2(r_{13}r_{24}+r_{14}r_{23})\leq (2K)^2\max\{r_{13}r_{24},\, r_{14}r_{23}\}.
$$
\end{lemma} \begin{proof} We can assume, without loss of
generality, that $r_{13}$ is the smallest of the numbers $r_{13}, r_{14}, r_{24}, r_{23}$, and that $r_{23}\ge r_{14}$. Clearly, it
suffices to show that
$$
r_{12}r_{34}\leq K^2(r_{13}r_{24}+r_{14}r_{23}+2\sqrt{r_{13}r_{24}r_{14}r_{23}}).
$$
Equivalently, we need to show that $\alpha\geq 0$, where
$$
\alpha=-r_{12}r_{34}+K^2(r_{13}r_{24}+r_{14}r_{23}+2\sqrt{r_{13}r_{24}r_{14}r_{23}}).
$$
By the assumptions we have
$r_{12}\leq K\min\{r_{13}+r_{23},\, r_{14}+r_{24}\}$ and $r_{34}\leq K\min\{r_{13}+r_{14},\, r_{23}+r_{24}\}$.
If $r_{14}+r_{24}\leq r_{13}+r_{23}$, then $r_{23}\geq r_{14}+r_{24}-r_{13}$. Since $r_{24}\geq r_{13}$, we obtain
\begin{equation*}
\begin{split}
\alpha \geq -K^2(r_{14}+r_{24})(r_{13}+r_{14}) &+K^2\big(r_{13}r_{24}+r_{14}(r_{14}+r_{24}-r_{13})+\\
&+2\sqrt{r_{13}r_{24}r_{14}(r_{14}+r_{24}-r_{13})}\big)\\ 
&=2K^2(\sqrt{r_{13}r_{24}r_{14}(r_{14}+r_{24}-r_{13})}-r_{13}r_{14})\geq 0.
\end{split}
\end{equation*}

Now suppose that $r_{14}+r_{24}\geq r_{13}+r_{23}$.
Then $r_{23}\leq r_{14}+r_{24}-r_{13}$ and hence
$$
\alpha\geq -K^2(r_{13}+r_{23})(r_{13}+r_{14})+K^2\big(r_{13}r_{24}+r_{14}r_{23}+2\sqrt{r_{13}r_{24}r_{14}r_{23}}\big)=K^2f(r_{23}), 
$$
where
$$
f(x)=r_{13}r_{24}+2\sqrt{r_{13}r_{24}r_{14}}\sqrt{x}-(r_{13})^2-r_{13}r_{14}-r_{13}x.
$$
The function $f(x)$
is increasing on the interval $[r_{14}, r_{14}+r_{24}-r_{13}]$.
Indeed, for each $x\in [r_{14}, r_{14}+r_{24}-r_{13}]$ we have
$r_{13}x-r_{24}r_{14}\leq
r_{13}(r_{14}+r_{24}-r_{13})-r_{24}r_{14}=(r_{14}-r_{13})(r_{13}-r_{24})\leq 0$ and hence $r_{13}\sqrt{x}-\sqrt{r_{13}r_{24}r_{14}}\leq 0$.
The latter is equivalent to $f^\prime(x)\geq 0$. Since $f(r_{14})=r_{13}r_{24}+2r_{14}\sqrt{r_{13}r_{24}}-(r_{13})^2-2r_{13}r_{14}
=r_{13}(r_{24}-r_{13})+2r_{14}(\sqrt{r_{13}r_{24}}-r_{13})\geq 0$, we obtain $\alpha\geq K^2f(r_{23})\geq K^2f(r_{14})\geq 0$,
completing the proof of the first part. Since $(a+b)^2\leq 2(a^2+b^2)$ for all real numbers $a$ and $b$, the second part follows.
\end{proof}

Next, we define a distance function $\mu_P\colon X\times X\to [0,+\infty)$ by
\begin{equation}\label{2.1}
\mu_P(x,y)=\prod_{i=1}^{k}\mu_{p_i}(x,y)=\prod_{i=1}^{k}\big[d(x,y)+\sqrt{d(x,p_i)d(y,p_i)}\big].
\end{equation}

\begin{lemma}\label{lem3}
For all $x,y,z\in X$ we have
$$
\mu_P(x,y)\leq \Big(\frac {27}{2}\Big)^k\Big(\mu_P(x,z)+\mu_P(z,y)\Big).
$$
Moreover,
$$
\mu_P(x,y)\mu_P(z,w)\leq 4\Big(\frac {27}{2}\Big)^{2k}\max\Big\{\mu_P(x,z)\mu_P(y,w),\,\, \mu_P(x,w)\mu_P(y,z)\Big\}.
$$
\end{lemma}
\begin{proof}
Using (\ref{2.1}) and Lemma~\ref{L3} we have
\begin{equation*}
\begin{split}
\mu_P(x,y)&=\prod_{i=1}^{k}\mu_{p_i}(x,y)\leq\Big(\frac {3}{2}\Big)^k\,\prod_{i=1}^{k}\Big(\mu_{p_i}(x,z)+\mu_{p_i}(y,z)\Big)\\
&\leq\Big(\frac {3}{2}\Big)^k\,{9}^k\,\Big(\prod_{i=1}^{k}\mu_{p_i}(x,z)+\prod_{i=1}^{k}\mu_{p_i}(y,z)\Big)\\
&=\Big(\frac {27}{2}\Big)^k\,\Big(\mu_P(x,z)+\mu_P(y,z)\Big),
\end{split}
\end{equation*}
completing the proof of the first part. The second part follows from the first part and Lemma~\ref{lem4}.
\end{proof}

%%%%%%%%%%%%%%%%%%%%%%%%%%%%%%%%%%%%%%%%%%%%%%%%%%%%%%%%%%%%%%%%%%%%%%%%%%%%%%%%%

\section{Gromov hyperbolicity of the average of one-point, scale-invariant Cassinian metrics}\label{Gromov Hyperbolic}

We begin by showing that each one-point, scale-invariant Cassinian metric is Gromov hyperbolic. Recall that a metric space $(X,d)$ is Gromov hyperbolic if 
\begin{equation}\label{four-point hyperbolicity condition}
d(x,y)+d(z,v)\leq \big[d(x,z)+d(y,v)\big]\vee \big[d(x,v)+d(y,z)\big]+2\delta
\end{equation}
for all $v,x,y,z\in X$ and for some $\delta\geq 0$. The reader is referred to (\cite{Vai1}) for a detailed discussion Gromov hyperbolic metric spaces.
Recall that
$$
\tilde\tau_p(x,y)\leq\tau_p(x,y)\leq\tilde\tau_p(x,y)+\log 2
$$
for all $x,y\in X\setminus\{p\}$ (see, (\ref{aaaEq2.2})). It follows that if the metric $\tilde\tau_p$ satisfies (\ref{four-point hyperbolicity condition}) with a constant $\delta$, then the metric $\tau_p$ satisfies (\ref{four-point hyperbolicity condition}) with a constant $\delta+\log 2$.

\begin{lemma}\label{one-point}
Let $(X,d)$ be an arbitrary metric space and let $p\in X$ be any point. Then the space $(X\setminus\{p\},\,\tilde\tau_p)$ is Gromov hyperbolic with $\delta=\log 3$. In particular, the space $(X\setminus\{p\},\,\tau_p)$ is Gromov hyperbolic with $\delta=\log 3+\log 2$.
\end{lemma}
\begin{proof} It suffices to show that $\tilde\tau_p$ satisfies (\ref{four-point hyperbolicity condition}) with $\delta=\log 3$. Let $x,y,z,v\in X\setminus\{p\}$ be arbitrary points.
By Lemma~\ref{lem0} we have 
$$
\mu_p(x,y)\mu_p(z,v)\leq 9\big[\mu_p(x,z)\mu_p(y,v)\vee\mu_p(x,v)\mu_p(y,z)\big]
$$ 
or, equivalently,
\begin{equation*}
\begin{split}
&\frac {\mu_p(x,y)\mu_p(z,v)}{\sqrt{d(x,p)d(y,p)d(z,p)d(v,p)}}\\
&\leq 9\Bigg[\frac {\mu_p(x,z)\mu_p(y,v)}{\sqrt{d(x,p)d(y,p)d(z,p)d(v,p)}}\vee\frac {\mu_p(x,v)\mu_p(y,z)}{\sqrt{d(x,p)d(y,p)d(z,p)d(v,p)}}\Bigg].
\end{split}
\end{equation*}
The latter implies
\begin{equation}\label{bEq2.2}
\tilde\tau_p(x,y)+\tilde\tau_p(z,v)\leq\Big[\tilde\tau_p(x,z)+\tilde\tau_p(y,v)\Big]\vee \Big[\tilde\tau_p(x,v)+\tilde\tau_p(y,z)\Big]+2\log 3,
\end{equation}
completing the proof.
\end{proof}

We are now ready to present the main result of the paper. Let $(X,d)$ be any metric space and let $p_1,p_2,\dots,p_k$ be any points in $X$. Put $P=\{p_1,p_2,\dots,p_k\}$ and 
$D=X\setminus\{p_1,p_2,\dots,p_k\}$. We define a new metric $\hat\tau_D$ on $D$ by taking the simple average of the one-point, scale-invariant Cassinian metrics $\tau_{p_i}$, $i=1,2,\dots,k$. Namely, for $x,y\in D$ we define
\begin{equation}\label{3.1}
\hat\tau_D(x,y)=\frac {1}{k}\big[\tau_{p_1}(x,y)+\tau_{p_2}(x,y)+\cdots+\tau_{p_k}(x,y)\big]=\frac {1}{k}\sum_{i=1}^{k}\tau_{p_i}(x,y).
\end{equation}
It is clear that the average of any finitely many metrics is again a metric. We have
\begin{equation}\label{tilde and hat}
\tilde\tau_D(x,y)\leq\hat\tau_D(x,y)\leq\tilde\tau_D(x,y)+\log 2
\end{equation}
for all $x,y\in D$, where
\begin{equation}\label{tilde definition}
\tilde\tau_D(x,y)=\frac {1}{k}\sum_{i=1}^{k}\tilde\tau_{p_i}(x,y)=\frac {1}{k}\log\Bigg(\prod_{i=1}^{k}\frac {\mu_{p_i}(x,y)}{\sqrt{d(x,p_i)d(y,p_i)}}\Bigg).
\end{equation}

\begin{theorem}\label{T1}
The space $(D,\hat\tau_D)$ is Gromov hyperbolic with $\delta= 3\log 3+\log 2$. In particular, if $(X,d)$ is Ptolemaic, then the space $(D,\tilde\tau_D)$ is Gromov hyperbolic with $\delta= 3\log 3$.
\end{theorem}
\begin{proof}
It suffices to show that for all $x,y,z,w\in D$ we have
$$ 
\tilde\tau_D(x,y)+\tilde\tau_D(z,w)\leq \max\big\{\tilde\tau_D(x,z)+\tilde\tau_D(y,w),\, \tilde\tau_D(x,w)+\tilde\tau_D(y,z)\big\}+6\log 3.
$$
Using Lemma~\ref{lem3} we obtain
\begin{equation*}
\begin{split}
&\tilde\tau_D(x,y)+\tilde\tau_D(z,w)=\frac {1}{k}\log\Bigg(\prod_{i=1}^{k}\frac {\mu_{p_i}(x,y)\mu_{p_i}(z,w)}{\sqrt{d(x,p_i)d(y,p_i)d(z,p_i)d(w,p_i)}}\Bigg)\\
&=\frac {1}{k}\log\Bigg(\frac {\prod_{i=1}^{k}\mu_{p_i}(x,y)\prod_{i=1}^{k}\mu_{p_i}(z,w)}{\prod_{i=1}^{k}\sqrt{d(x,p_i)d(y,p_i)d(z,p_i)d(w,p_i)}}\Bigg)\\
&=\frac {1}{k}\log\Bigg(\frac {\mu_P(x,y)\mu_P(z,w)}{\prod_{i=1}^{k}\sqrt{d(x,p_i)d(y,p_i)d(z,p_i)d(w,p_i)}}\Bigg)\\
& \leq\frac {1}{k}\log\Bigg(\frac {4(27/2)^{2k}\max\big\{\mu_P(x,z)\mu_P(y,w),\,\, \mu_P(x,w)\mu_P(y,z)\big\}}{\prod_{i=1}^{k}\sqrt{d(x,p_i)d(y,p_i)d(z,p_i)d(w,p_i)}}\Bigg)\\
&=\frac {1}{k}\log\Bigg(\frac {\max\big\{\mu_P(x,z)\mu_P(y,w),\,\, \mu_P(x,w)\mu_P(y,z)\big\}}{\prod_{i=1}^{k}\sqrt{d(x,p_i)d(y,p_i)d(z,p_i)d(w,p_i)}}\Bigg)+2\log (27/2)+\frac {1}{k}\log 4\\
&=\max\big\{\tilde\tau_D(x,z)+\tilde\tau_D(y,w),\, \tilde\tau_D(x,w)+\tilde\tau_D(y,z)\big\}+2(\log (27/2)+\frac {1}{k}\log 2)\\
&\leq\max\big\{\tilde\tau_D(x,z)+\tilde\tau_D(y,w),\, \tilde\tau_D(x,w)+\tilde\tau_D(y,z)\big\}+6\log 3,
\end{split}
\end{equation*}
completing the proof.
\end{proof}

\begin{definition}
In the context of a general metric space $(X,d)$, the metric $\hat\tau_D$ will be referred to as the average scale-invariant Cassinian metric. 
\end{definition}

We end the paper with the following example that shows that the sum of two Gromov hyperbolic metrics is not, in general, Gromov hyperbolic. Consider the two-dimensional Euclidean space $\mathbb R^2$ equipped with the Euclidean metric $|-|$. For $x\in\mathbb R^2$ we write $x=(x_1,x_2)$. Define metrics $d_1$ and $d_2$ on $\mathbb R^2$ by
\begin{equation*}
d_1(x,y)=|x_1-y_1|+\tan^{-1}(|x_2-y_2|)\quad\text{and}\quad d_2(x,y)=|x_2-y_2|+\tan^{-1}(|x_1-y_1|).
\end{equation*}
Clearly, both $d_1$ and $d_2$ are non-negative and symmetric, and $d_m(x,y)=0$ ($m=1,2$) if and only if $x=y$. Since $\tan^{-1}$ is increasing and concave function on $[0,\infty)$, we see that both $d_1$ and $d_2$ obey the triangle inequality. Thus, $d_1$ and $d_2$ are indeed metrics on $\mathbb R^2$.

\begin{lemma}\label{sum of metrics is not GH}
The spaces $(\mathbb R^2, d_1)$ and $(\mathbb R^2, d_2)$ are Gromov hyperbolic with $\delta=\pi/2$, but the space $(\mathbb R^2, d)$, $d=d_1+d_2$, is not Gromov hyperbolic.
\end{lemma}
\begin{proof}
Due to similarity between $d_1$ and $d_2$ it is enough to show that $(\mathbb R^2, d_1)$ is Gromov hyperbolic with $\delta=\pi/2$. First, observe that the Euclidean distance on $\mathbb R$ is Gromov hyperbolic with $\delta=0$. That is, for all $p,q,r,s\in\mathbb R$, we have
\begin{equation}\label{EuclideanZeroHyperbolic}
|p-q|+|r-s|\leq\big[|p-r|+|q-s|\big]\vee\big[|p-s|+|q-r|\big].
\end{equation}

Let $x=(x_1,x_2)$, $y=(y_1,y_2)$, $z=(z_1,z_2)$, and $v=(v_1,v_2)$ be arbitrary points in $\mathbb R^2$. Using (\ref{EuclideanZeroHyperbolic}) along with the fact that $\tan^{-1}(a)<\pi/2$ for all $a\in [0,+\infty)$, we obtain
\begin{equation*}
\begin{split}
d_1(x,y)+d_1(z,v)&=|x_1-y_1|+|z_1-v_1|+\tan^{-1}(|x_2-y_2|)+\tan^{-1}(|z_2-v_2|) \\
&\leq |x_1-y_1|+|z_1-v_1|+\frac {\pi}{2}+\frac {\pi}{2}\\
&\leq \big[|x_1-z_1|+|y_1-v_1|\big]\vee\big[|x_1-v_1|+|y_1-z_1|\big]+2\cdot\frac {\pi}{2}\\
&\leq \big[d_1(x,z)+d_1(y,v)\big]\vee\big[d_1(x,v)+d_1(y,z)\big]+2\cdot\frac {\pi}{2},
\end{split}
\end{equation*}
completing the proof of the first part.

Next, we show that $(\mathbb R^2, d)$ is not Gromov hyperbolic. Observe that $d$ is roughly similar to the {\em taxicab} metric. That is,
\begin{equation}\label{taxi}
d_T(x,y)\leq d(x,y)\leq d_T(x,y)+\pi\qquad\text{for\ all}\qquad x,y\in\mathbb R^2.
\end{equation} 
Here $d_T$ is the taxicab metric defined by $d_T(x,y)=|x_1-y_1|+|x_2-y_2|$. 
It is known that the taxicab metric is not Gromov hyperbolic. Indeed, for $t>0$ and 
$$
x=(0,0),\quad y=(t,t),\quad z=(0,t),\quad v=(t,0)
$$ 
we have
$$
d_T(x,y)+d_T(z,v)=2t,\quad d_T(x,z)+d_T(y,v)=t\quad\text{and}\quad d_T(x,v)+d_T(y,z)=t.
$$
Hence there exists no $\delta\geq 0$ such that
$$
d_T(x,y)+d_T(z,v)\leq\big[d_T(x,z)+d_T(y,v)\big]\vee\big[d_T(x,v)+d_T(y,z)\big]+2\delta
$$
for all $t>0$. Finally, it follows from (\ref{taxi}) that the space $(\mathbb R^2, d)$ is not Gromov hyperbolic, completing the proof.
\end{proof}

%%%%%%%%%%%%%%%%%%%%%%%%%%%%%%%%%%%%%%%%%%%%%%%%%%%%%%%%%%%%%%%%%%%%%%%%%%%%%%%%%
%%%%%%%%%%%%%%%%%%%%%%%%%%%%%%%%%%%%%%%%%%%%%%%%%%%%%%%%%%%%%%%%%%%%%%%%%%%%%%%%%%%%%%%%%%%%%%

\end{document}